\documentclass[11pt]{article}
\usepackage{amssymb,amsmath,amsthm,graphicx}
\usepackage{float}
\usepackage{color,enumerate}
\usepackage{fullpage}

\newtheorem{theorem}{Theorem}

\newtheorem{corollary}[theorem]{Corollary}
\newtheorem{proposition}[theorem]{Proposition}
\newtheorem{question}{Question}

\newcommand{\barva}[1]{{\color{black}#1}}

\theoremstyle{definition}

\newcommand{\gmt}{\gamma_t}
\newcommand{\GcH}{G\Box H}
\newcommand{\GcK}{G\Box K_2}
\newcommand{\gmtG}{\gamma_t(G)}
\newcommand{\gmtH}{\gamma_t(H)}
\newcommand{\gmtGH}{\gamma_t(\GcH)}
\newcommand{\gmtGK}{\gamma_t(\GcK)}
\newcommand{\oneN}{\{ 1,\ldots,n \}}
\newcommand{\iOneN}{i \in \oneN}

\title{On total domination in the Cartesian product of graphs}

\author{
Bo\v{s}tjan Bre\v{s}ar$^{a,b}$\thanks{Email: \texttt{bostjan.bresar@um.si}}
\and
Tatiana Romina Hartinger$^{c,d}$\thanks{Email: \texttt{tatiana.hartinger@iam.upr.si}}
\and
Tim Kos$^{b}$\thanks{Email: \texttt{tim.kos@imfm.si}}
\and
Martin Milani\v c$^{c,d}$\thanks{Email: \texttt{martin.milanic@upr.si}}
}

\begin{document}
\maketitle

\begin{center}
$^a$ Faculty of Natural Sciences and Mathematics, University of Maribor, Slovenia\\

$^b$ Institute of Mathematics, Physics and Mechanics, Ljubljana, Slovenia\\

$^c$ University of Primorska, UP IAM, Muzejski trg 2, SI 6000 Koper, Slovenia\\

$^d$ University of Primorska, UP FAMNIT, Glagolja\v ska 8, SI 6000 Koper, Slovenia \\
\end{center}

\bigskip

\begin{abstract}
Ho proved in [A note on the total domination number, Util.~Math.~77 (2008) 97–-100] that the total domination number of the Cartesian product
of any two graphs without isolated vertices is at least one half of the product of their total domination numbers.
We extend a result of Lu and Hou from~[Total domination in the Cartesian product of a graph and $K_2$ or $C_n$, Util.~Math.~83 (2010) 313–-322] by characterizing the pairs of graphs $G$ and $H$ for which $\gmtGH=\frac{1}{2}\gmtG \gmt (H)\,$, whenever $\gmtH=2$. In addition, we present an infinite family of graphs $G_n$ with $\gmt (G_n)=2n$, which asymptotically approximate the equality in \hbox{$\gmt(G_n\Box G_n)\ge \frac{1}{2}\gmt(G_n)^2$}.
 \end{abstract}

\noindent
{\bf Keywords:} total domination, Cartesian product, total domination quotient  \\

\noindent
{\bf AMS subject classification (2010)}: 05C69

\section{Introduction}

\barva{All graphs considered in this paper are finite, simple, and undirected. A {\em dominating set} of a graph $G$ is a set $D \subseteq V(G)$ such that every vertex not in $D$ is adjacent to at least one vertex from $D$. The {\em domination number} of $G$, denoted by $\gamma(G)$, is the minimum cardinality of a dominating set of $G$. A {\em total dominating set}, abbreviated a TD-set, of a graph $G$ with no isolated vertices, is a set $S$ of vertices of $G$ such that every vertex in $G$ is adjacent to a vertex from $S$. The {\em total domination number} of $G$, denoted by $\gmtG$, is the minimum cardinality of a TD-set of $G$. A TD-set of $G$ of cardinality $\gmtG$ will be referred to as a $\gmtG$-set. Given graphs $G$ and $H$, the {\em Cartesian product} $\GcH$ is the graph with the vertex set $V(G)\times V(H)$ in which two vertices $(u_1,v_1)$ and $(u_2,v_2)$ are adjacent if and only if either $u_1 = u_2$ and $v_1v_2 \in E(H)$ or $v_1 = v_2$ and $u_1u_2 \in E(G)$.}

Domination parameters in graph products have been given a lot of attention, which is largely due to the intriguing, long-lasting Vizing's conjecture on the domination number of the Cartesian products of graphs, see a recent survey~\cite{brdo-2012}. A related question on the total domination number of the Cartesian product of graphs was posed by Henning and Rall in~\cite{hr-2005}, asking whether the product of the total domination numbers of two graphs without isolated vertices is bounded above by twice the total domination number of their Cartesian product. The question was answered in the positive by Pak Tung Ho.

\begin{theorem}[Ho~\cite{ho_note}]
	\label{thm:ho_ineq}
	For graphs $G$ and $H$ without isolated vertices,
	$$\gmtG \gmt (H) \leq 2 \gmtGH\,.$$
\end{theorem}

The bound in Theorem \ref{thm:ho_ineq} is sharp as may be seen by taking $G=K_2$ and $H=K_n$. However it remains an open problem to characterize pairs of graphs $G$ and $H$ that achieve equality in the bound of Theorem~\ref{thm:ho_ineq}.

Henning and Rall characterized in~\cite{hr-2005} the trees $G$ such that $\gmtG\gmt (H) = 2\gmtGH$ holds for some graph $H$, while Lu and Hou did the same for the cycles~\cite{lu_hou_k2}. In all these cases the other factor $H$ is $kK_2$, i.e., a graph obtained by taking the disjoint union of $k$ copies of $K_2$. In the same paper Lu and Hou characterized the class of graphs $G$ with $\gmtG = \gmtGK$ \cite{lu_hou_k2}.
\barva{To explain their result, we need to introduce some more notation. The {\em neighborhood} of a vertex $v\in V(G)$ is the set $N_G(v)=\{u\in V(G)\,:\,uv\in E(G)\}$, while {\em neighborhood of a set} $X\subset V(G)$ is defined as $N_G(X)= \bigcup_{v\in X}N_G(v)$. (Hence, $X$ is a TD-set of $G$ if and only if $N_G(X)=V(G)$.) Given a graph $G$ and a set $X\subseteq V(G)$, the {\em subgraph of $G$ induced by $X$} is the graph denoted by $G[X]$ with vertex set $X$ and edge set $\{uv\in E(G): u,v\in X\}$.}
Lu and Hou defined the following families of graphs:
\begin{itemize}
  \item ${\mathcal{F}}_1$: the graphs $G$ such that $\gmtG = 2 \gamma (G)$,
  \item ${\mathcal{F}}_2$: the graphs $G$ that have a $\gmtG$-set $D$ that can be partitioned into two nonempty subsets $D_1$ and $D_2$ such that $D_1 = V(G) \setminus N_G(D_2)$ and $D_2 = V(G) \setminus N_G(D_1)$, and
  \item $\mathcal{F}_3$: the graphs $G$ whose vertex set $V(G)$ can be partitioned into two nonempty subsets $V_1$ and $V_2$ such that $G_1 = G[V_1] \in \mathcal{F}_1$, $G_2 = G[V_2] \in \mathcal{F}_2$, and $\gmtG = \gmt(G_1) + \gmt(G_2)$.
\end{itemize}

\barva{It is well-known and easy to see that for any graph $G$ with no isolated vertices, $\gamma_t(G)\le 2\gamma(G)$. Therefore, family ${\mathcal{F}}_1$
consists precisely of the graphs attaining equality in the above inequality.
The problem of characterizing the graphs of family ${\mathcal{F}}_1$ is open in general, see e.g.~\cite[Problem 4]{henning09}, while a constructive characterization of trees $T$ such that $\gamma_t(T) = 2\gamma(T)$ was given in~\cite{henning01}, see also~\cite{henning_book}.} In Fig.~\ref{fig:examples}, three separating examples for these classes are exhibited.

\begin{figure}[!h]
\begin{center}
\includegraphics[scale=0.75]{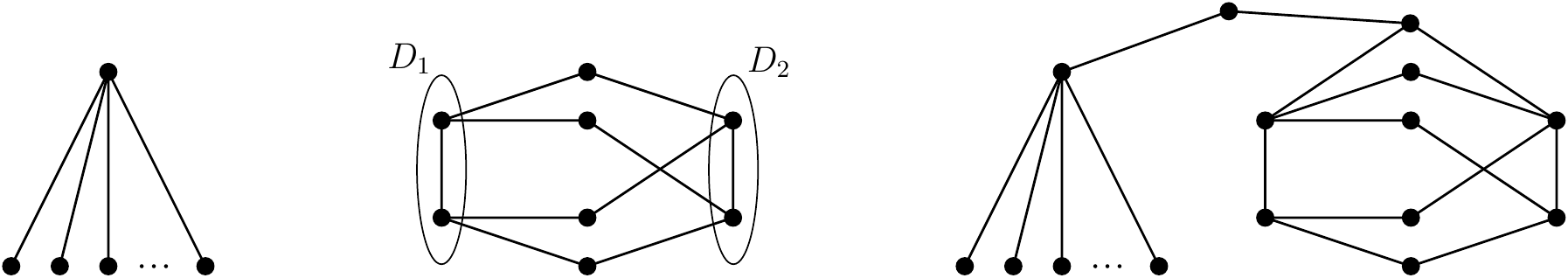}
\caption{A graph in ${\mathcal{F}}_1 \setminus ({\mathcal{F}}_2 \cup {\mathcal{F}}_3)$, a graph in ${\mathcal{F}}_2 \setminus ({\mathcal{F}}_1 \cup {\mathcal{F}}_3)$, and a graph in ${\mathcal{F}}_3 \setminus ({\mathcal{F}}_1 \cup {\mathcal{F}}_2)$, respectively.}
\label{fig:examples}
\end{center}
\end{figure}

\begin{theorem}[Lu and Hou \cite{lu_hou_k2}]
	\label{thm:lu_hou_k2}
	If $G$ is a graph without isolated vertices, then
	$\gmtG = \gmtGK$ if and only if
	$G \in \mathcal{F}_1 \cup \mathcal{F}_2\cup \mathcal{F}_3$.
\end{theorem}

In this paper we extend this result by showing that if $H$ is a connected graph with $\gamma_t(H)=2$ and $\gmtGH=\gmtG$ for some graph $G$, then, either $H$ is  isomorphic to $K_2$, and so $G$ is one of the graphs characterized in Theorem~\ref{thm:lu_hou_k2}, \barva{or $H$ is a graph with $\gamma(H)=1$, and $G$ is isomorphic to $K_2$} (see Section~\ref{sec:charK2}). In Section~\ref{sec:asymptotic} we continue the investigation of the existence of pairs $G$ and $H$ that achieve equality in the bound from Theorem~\ref{thm:ho_ineq}. While we found no other such pairs, we present a family of graphs $G_n$ \barva{with arbitrarily large total domination numbers} for which the ratio $\gamma_t(G_n\Box G_n)/(\gamma_t(G_n))^2$ is as close to $1/2$ as desired. \barva{Finally, in Section~\ref{sec:quot} we give several remarks on the {\em total domination quotient} $\gmtGH/(\gmtG \gmt(H))$ and propose a further study of this notion.} 

We conclude this section by giving some definitions and notation used in the rest of the paper.
\barva{If $X$ and $Y$ are subsets of vertices in $G$, then $X$ {\em totally dominates} $Y$ in $G$ if $Y\subseteq N_G(X)$. Similarly, we say that a vertex $x$  {\em totally dominates} a vertex $y$ if $xy\in E(G)$. Given a set $X\subseteq V(G)$ and a vertex $u\in X$, we define ${\it pn}_G(u,X)$ as the set $\{w\in V(G)\,:\, N_G(w)\cap X = \{u\}\}$. A member of the set ${\it pn}_G(u,X)$ is said to be an {\em $X$-private neighbor of $u$ in $G$}. The indices in the notions defined throughout this section will be omitted when the graph to which they refer will be clear from the context. Given two graphs $G$ and $H$ and a} vertex $h\in V(H)$, the set $G^h=\{(g,h)\in V(G\Box H)\,:\, g\in V(G)\}$ is called {\em a $G$-fiber} in the Cartesian product of $G$ and $H$. For $g\in V(G)$, the $H$-\emph{fiber} $^g\!H$ is defined as $^g\!H=\{(g,h)\in V(G\Box H)\,:\,\ h\in V(H)\}$. We may consider $G$-fibers and $H$-fibers as induced subgraphs when appropriate. The {\em projection} to $G$ is the map $p_G : V(\GcH)\rightarrow V(G)$
defined by $p_G(g,h) = g$.

\section{\barva{Pairs of graphs with $\gmtH = 2$ attaining equality in Ho's bound}}
\label{sec:charK2}

By Theorem~\ref{thm:ho_ineq}, any two graphs $G$ and $H$ without isolated vertices satisfy $\gmtG \gmt (H) \leq 2 \gmtGH$. \barva{In this section, we characterize pairs of graphs $G$ and $H$ without isolated vertices and with $\gmt (H) = 2$ such that $\gmtG \gmt (H) = 2 \gmtGH$, or, equivalently,
pairs of graphs $G$ and $H$ without isolated vertices such that $\gmt (H) = 2$ and $\gmtG = \gmtGH$.
Note that since for any graph $H$ without isolated vertices we have $\gmt(H)\ge 2$, inequality $\gmtG \gmt (H) \leq 2 \gmtGH$ 
implies \begin{equation}
\label{e1}
\gmtG \leq \gmtGH
\end{equation}
\noindent for any two graphs $G$ and $H$ with no isolated vertices, and equality is possible only when $\gmt(H)=2$. 
Thus, we will in fact characterize the pairs of graphs that achieve equality in~\eqref{e1}.}

Let $G$ and $H$ be connected graphs such that $\gmt(H)=2$ and $\gmtG = \gmtGH$. Let $D$ be a $\gmtGH$-set, and
let $V(H)=\{h_1,\ldots,h_n\}$.
The set $D$ can be partitioned into two (possibly empty) subsets: $D' = \{ (g,h) \in D \ |\ \exists \ h' \in V(H)$ such that $h' \neq h$ and $(g,h') \in D \}$ and $D''= D \setminus D'$. Let $D''_i = \{ (g,h) \in D'' \ |\ h=h_i \}$ for every $\iOneN$. 

On the graph $G$ we define the sets \barva{$S=p_G(D)$}, $S' = p_G(D')$, $S'' = p_G(D'')$, \barva{$P=N_G(S)\setminus S$}, \barva{$P' = N_G(S')\setminus S'$}, $P'' = N_G(S'')\setminus S''$ and $S''_i = p_G(D''_i)$ for every $\iOneN$.

\begin{proposition}
	\label{thm:sets}
	Let $G$ and $H$ be connected graphs such that $\gmt(H)=2$ and $\gmtG = \gmtGH$.
Let $D, D', D'',$ $D''_i, S, S', S'', P, P', P''$ and $S''_i$ be the sets as defined above, where $\iOneN$. Then the following statements hold for $G$:
	\begin{enumerate}[{\bf (A)}]
		\item \label{thm:setsA} $S \cup P = V(G)$.
		\item \label{thm:setsB} $\gmtG = 2 |S'| + |S''|$.
		\item \barva{\label{thm:setsC} $S'$ is independent and no two vertices from $S'$ have a common neighbor.}
		\item \label{thm:setsD} There are no edges between $S'$ and $S''$.
		\item \label{thm:setsE} \barva{There exists a positive integer $k$ such that} $S''$ induces a $kK_2$.
	    \item \label{thm:setsE'} \barva{If $\barva{T'}$ is a minimum set of vertices that totally dominates $S'$, then $T' \subseteq P'$, $|\barva{T'}| = |S'|$, and $X = S' \cup \barva{T'} \cup S''$ is a $\gmtG$-set. In addition, for every vertex $g\in S''$, the set ${\it pn}_G(g,X)$ is a subset of $S''$ of size 1.}
		\item \label{thm:setsF} For every $\iOneN$, \barva{there exists a positive integer $k$ such that} $S''_i$ induces a $kK_2$.
	    \item \label{thm:setsG} For every $\iOneN$, $S''_i$ totally dominates $P'' \cup S''_i$.
	    \item \label{thm:setsH} For every $\iOneN$, no two vertices from $S''_i$ have a common neighbor.
	    \item \label{thm:setsI} No vertex from $S'$ has a common neighbor with a vertex from $S''$.
	    \item \label{thm:setsJ} $S', S'', P'$, and $P''$ are pairwise disjoint sets.
	   \item \label{thm:setsK} $\gmt (G[S' \cup P']) = 2 \gamma (G[S' \cup P' ]) \barva{= 2|S'|}$.
		\item \label{thm:setsL} $\gmt (G[S'' \cup P'']) = | S'' |$.
		\item \label{thm:setsM} If $H \neq K_2$, then $S' = \emptyset$.
\end{enumerate}
	
\end{proposition}

\begin{proof} \barva{To prove statement~{\bf (\ref{thm:setsA})}, consider an arbitrary vertex $g$ of $G$.
It suffices to show that $g\in P$ whenever $g\not\in S$. So let $g\in V(G)\setminus S$ and fix an arbitrary vertex $h\in V(H)$. 
Since $S = p_G(D)$ and $g\not\in S$, we have $(g,h)\notin D$. Since $D$ is a TD-set of $G\Box H$, there exists a vertex $(g',h')\in D$ adjacent to $(g,h)$ in $G\Box H$. Since $g'\neq g$, we must have $h' = h$ and $gg'\in E(G)$. This implies that $p_G(g',h)\in S$, hence $g\in P$. Therefore~{\bf (\ref{thm:setsA})} holds.}
	
To prove statement~{\bf (\ref{thm:setsB})},	first note that $|D'|\geq 2|S'|$ and $|D''|=|S''|$. Thus, $\gmtG = \gmtGH = |D| = |D'| + |D''|\geq 2 |S'| + |S''|$. \barva{It is clear that $S$ totally dominates $P$. We claim that $S$ also totally dominates $S''$. Indeed, if $(g,h)\in D''$, then any $(g,h')\in V(G\Box H)$, where $h'\neq h$, does not belong to $D$. Since $D$ is a TD-set of $G\Box H$, we infer that there exists $(g',h)\in D$ such that $g'g\in E(G)$, and so $g'\in S$ totally dominates $g$. Altogether, the fact that $S$ totally dominates $S''$ and using {\bf (\ref{thm:setsB})}, we infer that $S$ totally dominates $V(G)\setminus S'$; and to totally dominate $S'$ we need to add at most $|S'|$ vertices to the vertices of $S$. Hence, $\gmt(G)\le |S|+|S'| = 2 |S'| + |S''| $, and combining this with the inequality at the beginning of this paragraph, we get statement~{\bf (\ref{thm:setsB})}.
		
As noted in the previous paragraph, the only vertices that may not be totally dominated in $G$ by $S$ are the vertices in $S'$. To satisfy statement~{\bf (\ref{thm:setsB})}, adding less than $|S'|$ vertices to $S' \cup S''$ does not yield a TD-set of $G$. Because one needs to totally dominate $|S'|$ vertices by using at least $|S'|$ vertices, each vertex that we add to $S$ in order to obtain a TD-set, has to totally dominate exactly one vertex from $S'$. This readily implies statement~{\bf (\ref{thm:setsC})}. Next, if a vertex in $S'$ has a neighbor in $S''$, then this vertex is already totally dominated by $S$, which again yields a TD-set of $G$ with less than $2 |S'| + |S''|$ vertices, a contradiction. Hence statements~{\bf (\ref{thm:setsD})} holds.
}
	
\barva{To show that statements~{\bf (\ref{thm:setsE})} and {\bf (\ref{thm:setsE'})}   hold, consider a minimum set of vertices $T'$ in $G$ that totally dominates $S'$. Statements~{\bf (\ref{thm:setsC})} and~{\bf (\ref{thm:setsD})} imply that $T' \subseteq P'$ and $|T'| = |S'|$, which in turn, using statement~{\bf (\ref{thm:setsB})}, implies that $X = S' \cup \barva{T'} \cup S''$ is a $\gmtG$-set (proving the first assertion of statement~{\bf (\ref{thm:setsE'})}). 
Therefore every vertex in $X$ has an $X$-private neighbor in $G$. We first show that vertices of $S''$ can have $X$-private neighbors only in $S''$. By statement~{\bf (\ref{thm:setsD})} they have no neighbors in $S'$. Suppose for a contradiction that there is a vertex $g\in S''_i$ (for some $\iOneN$) that has an $X$-private neighbor $g'$ in $G$ such that $g'\in P$. Let $j\in \{1,\ldots, n\}\setminus \{i\}$. Since $g'\in P$, no vertex from the $H$-fiber $^{g'}\!H$ is in $D$. Thus, since $D$ totally dominates $(g',h_j)$, there is some $g''\in V(G)$ such that $g'g''\in E(G)$ and $(g'',h_j)\in D$. Note that $g''\neq g$ since $i\neq j$ and $g\in S''_i$. Since $g''\in S$, we have $g''\in X$ and therefore $N(g')\cap X\neq \{g\}$. This shows that vertices of $S''$ can have $X$-private neighbors only in $S''$, as claimed. It follows that for every vertex $g\in S''$, the set ${\it pn}_G(g,X)$ is a non-empty subset of $S''$ and every vertex $g'\in {\it pn}_G(g,X)$ is of degree 1 in the subgraph of $G$ induced by $S''$. Applying the same argument with $g'$ in place of $g$ shows that $S''$ induces a $kK_2$ for some integer $k$, proving statement {\bf (\ref{thm:setsE})}. In particular, every vertex $g\in S''$ has a unique $X$-private neighbor, namely its unique neighbor in the graph $G[S'']$. This proves the second assertion of statement {\bf (\ref{thm:setsE'})}.

Since $D$ is a total dominating set in $G\Box H$, each vertex $(g,h_i)$ of $D''_i$ has a neighbor in $D$, which is not in the same $H$-fiber as $(g,h_i)$ by definition of $D''$. Hence, there is a vertex $g'\in V(G)$ such that $(g',h_i)\in D$ is a neighbor of $(g,h_i)$, and so, by statements~{\bf (\ref{thm:setsD})} and~{\bf (\ref{thm:setsE})}
we have $(g',h_i)\in D_i''$. By projecting to $G$, we get that for each vertex $g$ in $S_i''$, its unique neighbor $g'$ in $S''$ is in $S_i''$. We infer that each $S_i''$ induces a graph isomorphic to some $kK_2$, i.e., statement~{\bf (\ref{thm:setsF})} holds.

To prove statement~{\bf (\ref{thm:setsG})} first note that by statement~{\bf (\ref{thm:setsF})} the set $S''_i$ totally dominates $S''_i$ for every $\iOneN$. For the purpose of getting a contradiction, let us assume that there exists a vertex $g\in P''$ that is not totally dominated by $S_i''$ for some $\iOneN$. On the one hand, since $g\in P''$, there exists a vertex $a\in S_j''$ for some $j, j\neq i$, such that $ag\in E(G)$. On the other hand, vertex $(g,h_i)$ must be totally dominated by some vertex in $D$, which can only be a vertex from $G^{h_i}$. Let $(b,h_i)\in D$ be a vertex that totally dominates $(g,h_i)$. Clearly, by our assumption that $g$ is not totally dominated by $S_i''$, we have $b\in S'$. We may assume without loss of generality that a minimum set of vertices $T'$, used in the definition of $X$ in statement~{\bf (\ref{thm:setsE'})}, that totally dominates $S'$ contains vertex $g$ (which totally dominates vertex $b$ from $S'$). But then $a$ has two neighbors in $X$, namely $g$, and the unique neighbor $a'$ in $S_j''$. Hence, as the only neighbor of $a'$ in $S$ is vertex $a$ (by statement~{\bf (\ref{thm:setsE})}), we infer that $a'$ has no $X$-private neighbors, a contradiction with statement~{\bf (\ref{thm:setsE'})}. This implies that $S''_i$ totally dominates $P''$, and so {\bf (\ref{thm:setsG})} holds.
}
	
Let $X = S' \cup T' \cup S''$, where $T'$ is a minimum set of vertices that totally dominates $S'$, be defined as in statement~{\bf (\ref{thm:setsE'})}, and recall that $X$ is a $\gmtG$-set. Suppose now that two distinct vertices $u,v \in S''_i$, where $\iOneN$, have a common neighbor $w$ in $G$. Statements~{\bf (\ref{thm:setsD})}, {\bf (\ref{thm:setsE})}, and~{\bf (\ref{thm:setsF})} imply that $w\in P''$. If $u$ and $v$ are adjacent, then by statements~{\bf (\ref{thm:setsE'})} and~{\bf (\ref{thm:setsF})} vertex $u$ is the only $X$-private neighbor of $v$ in $G$ and vice versa.
\barva{We will show that in this case $X'=(X - \{u,v\}) \cup \{w\}$ is also a TD-set of $G$. Indeed, by statement~{\bf (\ref{thm:setsG})}, $P''\cup S_j''$ is totally dominated by $S_j''$ for each $j\in\{1,\ldots,n\}\setminus\{i\}$; all vertices $x\in S''\setminus\{u,v\}$ are totally dominated by their unique neighbor in $S''$, while $u$ and $v$ are totally dominated by $w$. Every vertex in $G$ not in $S''\cup P''$ is totally dominated by a vertex in $X$, and every such vertex is also a vertex of $X'$. Hence $X'$ is a TD-set of $G$, which contradicts the minimality of $X$.}	
If $u$ and $v$ are non-adjacent, then $u$ has a neighbor $u' \in S''_i$ with ${\it pn}_G(u',X) = \{u\}$, and $v$ has a neighbor $v' \in S''_i$ with ${\it pn}_G(v',X) = \{v\}$. By a similar reasoning as in the previous case, we infer that $(X - \{u',v'\}) \cup \{w\}$ is also a TD-set of $G$, contradicting the minimality of $X$. Statement~{\bf (\ref{thm:setsH})} follows.
	
Suppose that $w$ is a common neighbor of $u \in S'$ and $v \in S''$. \barva{We may assume without loss of generality that $w$ is in $\barva{T'}$ (because vertices from $\barva{T'}$ were chosen  arbitrarily as neighbors of vertices from $S'$), and so $w\in X$.
Recall that $v$ has a neighbor $v' \in S''$ such that ${\it pn}_G(v',X) = \{v\}$, which is a contradiction with $w\in X$ being adjacent to $v$. Statement~{\bf (\ref{thm:setsI})} follows.}
	
Next we show that $S', S'', P'$ and $P''$ are pairwise disjoint sets. It is clear from definitions that 
{$S'\cap S'' = S'\cap P' = S'' \cap P'' = \emptyset$.}
As a consequence of statement~{\bf (\ref{thm:setsD})},  $S' \cap P'' = \emptyset$ and $S'' \cap P' = \emptyset$. Now, $P' \cap P'' = \emptyset$ follows from statement~{\bf (\ref{thm:setsI})}. Hence statement~{\bf (\ref{thm:setsJ})} is true.

Note that $S'\cup T'$ totally dominates $S' \cup P'$, and $S''$ totally dominates
$S'' \cup P''$. Since $\gmtG = 2 |S'| + |S''|$, we have $\gmt(G[S' \cup P']) \ge 2 |S'|$,
and $\gmt(G[S'' \cup P'']) \ge |S''|$. On the other hand, $S'$ is a dominating set of $G[S' \cup P' ]$, so 
$\gmt(G[S' \cup P' ])\le 2 \gamma(G[S' \cup P' ])\le 2 |S'|$, which implies that
$\gmt(G[S' \cup P' ])= 2 \gamma(G[S' \cup P' ])= 2 |S'|$ and proves statement~{\bf(\ref{thm:setsK})}. Since $S''$ is a total dominating set of $G[S'' \cup P'']$, we get $\gmt(G[S'' \cup P'']) \le |S''|$, hence $\gmt(G[S'' \cup P'']) = |S''|$, which proves statement~{\bf(\ref{thm:setsL})}.
	
From $|S''| = |D''|$ and $2 |S'| + |S''| = \gmtG = \gmtGH = |D| = |D'| + |D''|$ it follows that $|D'|=2|S'|$. Let $(g,h) \in D'$.
By the definition of $D'$, there is exactly one $h' \in H, h \neq h'$, such that $(g,h') \in D$. So if $H$ is a graph on $3$ or more vertices there exists some $\hat{h} \in H$ such that $(g,\hat{h}) \notin D$. \barva{Because of statements~{\bf (\ref{thm:setsC})} and {\bf (\ref{thm:setsD})} and the fact that $X$ is a $\gamma_t(G)$-set}, vertex $g$ has an $X$-private neighbor $g' \in P'$. But then $(g',\hat{h})$ is not totally dominated by any vertex of $D$.
 Hence, if $S' \neq \emptyset$, then $D' \neq \emptyset$, in which case $H = K_2$. This implies statement~{\bf (\ref{thm:setsM})}.
\end{proof}


From Proposition~\ref{thm:sets} one easily deduces Theorem~\ref{thm:lu_hou_k2}.
If $S''=\emptyset$, then $G$ is isomorphic to $G[S' \cup P']$,
and by statement~{\bf (\ref{thm:setsK})} we have $\gamma_t(G)=2\gamma(G)$, that is, $G\in {\cal F}_1$. If, on the other hand,  $S'=\emptyset$, then $G$ is isomorphic to $G[S'' \cup P'']$, which, by {statements~{\bf (\ref{thm:setsG})} and~{\bf (\ref{thm:setsL})}, belongs to
$\mathcal{F}_2$. In the remaining case, we infer, using statements~{\bf (\ref{thm:setsB})}, {\bf (\ref{thm:setsK})}, and {\bf (\ref{thm:setsL})}, that  $G \in \mathcal{F}_3$.

Now we have everything ready to derive the desired characterization.

\begin{theorem}\label{thm:GH2}
Let $G$ and $H$ be nontrivial connected graphs with $\gmt(H)=2$. Then $\gmtG = \gmtGH$ if and only if
\barva{one of the following conditions holds:}
\begin{enumerate}[(i)]
  \item \barva{$G = K_2$ and $\gamma(H) = 1$};
  \item $H=K_2$ and $G \in \mathcal{F}_1 \cup \mathcal{F}_2\cup \mathcal{F}_3$.
\end{enumerate}
\end{theorem}

\begin{proof}
Let $G$ and $H$ be connected graphs with $\gmtG = \gmtGH$, let $V(H)=\{h_1,\ldots,h_n\}$ and let
$D, D', D'',$ $D''_i, S', S'', P, P', P''$ and $S''_i$, for $\iOneN$,
be the sets as defined in the beginning of this section.
\barva{Suppose first that $G = K_2$. Then Theorem~\ref{thm:lu_hou_k2} (alternatively, Proposition~\ref{thm:sets}) implies that 
$H \in \mathcal{F}_1 \cup \mathcal{F}_2\cup \mathcal{F}_3$. If $H\in \mathcal{F}_2\cup \mathcal{F}_3$, then 
$\gamma_t(H) \ge 4$, which contradicts $\gmt(H)=2$. Therefore, $H\in \mathcal{F}_1$, that is, 
$\gmt(H)=2\gamma(H)$. Since $\gmt(H)=2$, we infer that $\gamma(H) = 1$ and condition $(i)$ holds.}

\barva{Suppose now that $G\neq K_2$.} If $H=K_2$, then the result follows from Theorem~\ref{thm:lu_hou_k2} (alternatively, from Proposition~\ref{thm:sets}). Otherwise, $H$ is a graph on $n$ vertices for some $n\ge 3$, and we will show that this will lead to a contradiction.
By statement~{\bf (\ref{thm:setsM})} in Proposition~\ref{thm:sets}, $S' = \emptyset$. Hence, $P' = \emptyset$ and $V(G) = S'' \cup P''$. By statement~{\bf (\ref{thm:setsL})} we get $\gmtG=|S''|$, and Proposition~\ref{thm:sets} also shows that $S''$ partitions into the sets $S''_i=p_G(D''_i)$, which are all non-empty. In addition, each $S''_i$ induces $kK_2$ for some $k\ge 1$, and totally dominates \barva{
$S''_i \cup P'' $}. Now, by \barva{statements~{\bf (\ref{thm:setsE})} and~{\bf(\ref{thm:setsF})}}, there are no edges between vertices of $S''_i$ and $S''_j$, for $i\neq j$. Since $G$ is connected and $n>1$, set $P''$ is nonempty, so let $x\in P''$.
By statements~{\bf (\ref{thm:setsG})} and~{\bf (\ref{thm:setsH})}, for every $\iOneN$, vertex $x$ has a unique neighbor in $S''_i$.
For $i\in \{1,2\}$, let $x_i$ be the unique neighbor of $x$ in $S''_i$ and let $y_i$ be the unique neighbor of $x_i$ in $S''_i$.
Let $T = (S''\setminus\{y_1,y_2\})\cup\{x\}$.
We claim that $T$ is a TD-set of $G$, which will imply that $\gmtG \le |T|<|S''|$, contrary to the optimality of $S''$.
First of all, set $P''$ is totally dominated by $T$, since $S''_3\subseteq T$.
For every $i\ge 3$, set $S''_i$ is totally dominated by itself
(and therefore by $T$, since $S''_i\subseteq T$).
Vertices $x_1$ and $x_2$ are totally dominated by $x$ (and therefore by $T$, since $x\in T$).
Vertices $y_1$ and $y_2$ are totally dominated by $x_1$ and $x_2$, respectively (and therefore by $T$, since $\{x_1,x_2\}\subseteq T$).
Any other vertex in $S''_i$, where $i\in \{1,2\}$, has a unique neighbor in $S''_i$, which belongs to $T$, and is therefore totally dominated by $T$.
This completes the proof.

The converse follows from Theorems~\ref{thm:ho_ineq} and~\ref{thm:lu_hou_k2}.
For the sake of completeness, we briefly describe the construction. 
\barva{First note that condition~{\it (i)} is a special case of condition~{\it (ii)} with roles of $G$ and $H$ interchanged (indeed, $\gamma(H) = 1$ and $\gamma_t(H) = 2$ imply $H\in {\cal F}_1$). Hence, let us assume that condition {\it (ii)} holds, that is,} let $G$ and $H$ be graphs such that $H=K_2$ and $G \in \mathcal{F}_1 \cup \mathcal{F}_2\cup \mathcal{F}_3$.
By Theorem~\ref{thm:ho_ineq}, we have $\gmtGH\ge \gmtG$.
If $G \in \mathcal{F}_1$, then $\gmtG = 2 \gamma(G)$ and a TD-set of $\GcH$ of size $\gmtG$ can be obtained by taking a copy of a fixed minimum dominating set of $G$ in each of the two $G$-fibers in $\GcH$.
Suppose now that $G\in \mathcal{F}_2$. Then, $G$ has a $\gmtG$-set $D$ that can be partitioned into two nonempty subsets
$D_1$ and $D_2$ such that $D_1 = V(G) \setminus N_G(D_2)$ and $D_2 = V(G) \setminus N_G(D_1)$.
A TD-set of $\GcH$ of size $\gmtG$ can be obtained by taking a copy of
$D_1$ in one of the two $G$-fibers in $\GcH$ and
$D_2$ in the other $G$-fiber.
Finally, suppose that $G\in \mathcal{F}_3$. Then, the vertex set of $G$ can be partitioned
into two nonempty subsets $V_1$ and $V_2$ such that
$G_1 = G[V_1] \in \mathcal{F}_1$,
$G_2 = G[V_2] \in \mathcal{F}_2$, and $\gmtG = \gmt(G_1) + \gmt(G_2)$.
Since $G_2\in \mathcal{F}_2$, graph $G_2$ has a $\gamma_t(G_2)$-set
$D''$ that can be partitioned into two nonempty subsets
$D_1''$ and $D_2''$ such that $D_1'' = V(G_2) \setminus N_{G_2}(D_2'')$ and $D_2'' = V(G_2) \setminus N_{G_2}(D_1'')$.
Moreover, let $D'$ be a minimum dominating set of $G_1$.
Then, a TD-set of $\GcH$ of size $\gmtG$ can be obtained by
taking a copy of $D'\cup D_1''$ in one of the two $G$-fibers in $\GcH$ and
$D'\cup D_2''$ in the other $G$-fiber.
\end{proof}

Note that since the class of graphs $\mathcal{F}_1 \cup \mathcal{F}_2 \cup \mathcal{F}_3$ is closed under disjoint union and under taking components,
the connectedness assumption on $G$ in Theorem~\ref{thm:GH2} could be replaced with the more general condition asserting that $G$ has no isolated vertices.

\section{Approximating the equality in $\gmtG \gmt (H) \le 2 \gmtGH$}
\label{sec:asymptotic}

For two connected graphs $G$ and $H$ without isolated vertices, let us consider the quotient of the
total domination number of their Cartesian product and the product of their total domination numbers,
$$q_t(G,H) = \frac{\gmtGH}{\gmtG \gmt(H)}\,.$$
We call it the {\em total domination quotient} of graphs $G$ and $H$.
By Theorem~\ref{thm:ho_ineq}, we infer that
\begin{equation}
\label{e2}
q_t(G,H)\ge 1/2
\end{equation}
for every two connected graphs $G$ and $H$ without isolated vertices.
For all known pairs of graphs with $q_t(G,H)= 1/2$, one of $G$ and $H$ is isomorphic to $K_2$
(see Theorem~\ref{thm:GH2}). We in fact suspect that there are no other such pairs.
On a related note, one may wonder whether there exists some $\epsilon>0$ such that for all connected graphs $G$ and $H$ having sufficiently large total domination numbers, we have $q_t(G,H)\ge 1/2+\epsilon\,.$
As we show next, this is not the case: we exhibit an infinite family of graphs $\{G_n\}_{n\ge 2}$ such that $\gamma_t(G_n) = 2n$ and
$\lim_{n\to\infty}q_t(G_n,G_n) = 1/2$. However, for each $n\ge 1$ we have $q_t(G_n,G_n) > 1/2$.

For $n\ge 1$, let $G_n$ denote the graph obtained from $K_n$ by attaching an end-vertex of a $P_3$ to each vertex of the $n$-clique.
Formally, this is a graph with vertex set $\{a_1,\ldots, a_n\}\cup \{b_1,\ldots, b_n\}\cup \{c_1,\ldots, c_n\}$,
where the set $\{a_1,\ldots, a_n\}$ forms a clique, every $b_i$ is adjacent to every $a_i$ and every $c_i$, and there are no other edges.
Note that $G_2\cong P_6$; see Fig.~\ref{G3G4} for the next two examples.

\begin{figure}[h!]
\begin{center}
\includegraphics[scale=0.6]{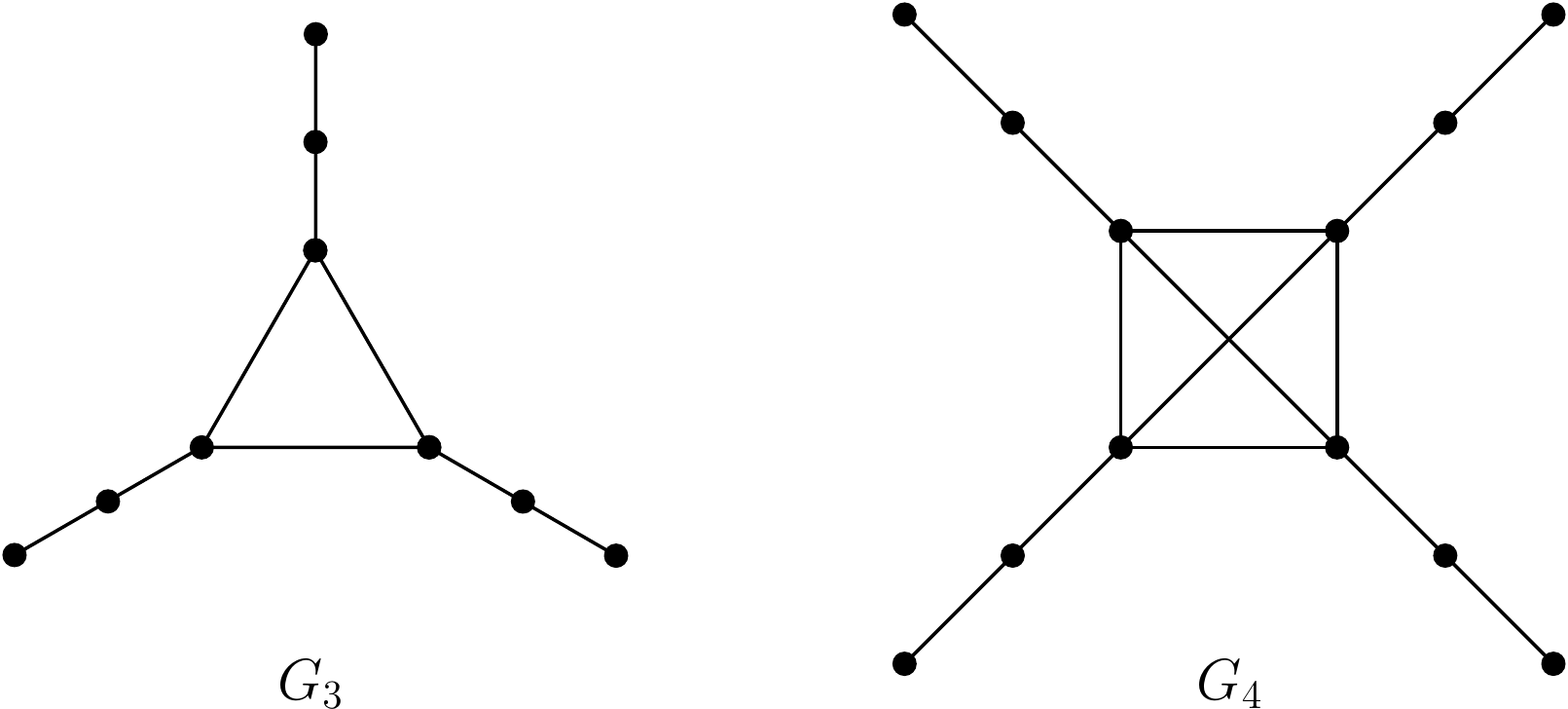}
\caption{The graphs $G_3$ and $G_4$.}
\label{G3G4}
\end{center}
\end{figure}


\begin{proposition}
For all integers $2\le k\le n$ we have
$$2kn+k\le
\gamma_t(G_k\Box G_n)\le 2kn+2k\,.$$
\end{proposition}

\begin{sloppypar}
\begin{proof}
Let us denote the vertices of the first factor, isomorphic to $G_k$, as
$a_1,\ldots, a_k, b_1,\ldots, b_k, c_1,\ldots, c_k$,
where $A = \{a_1,\ldots, a_k\}$ is the $k$-clique,
$B = \{b_1,\ldots, b_k\}$, $C = \{c_1,\ldots, c_k\}$ is the set of vertices
of degree $1$, and $b_i$ is the unique neighbor of $c_i$, for each $i$.
For the other factor, isomorphic to $G_n$, we will denote its vertices with
$x_1,\ldots, x_n, y_1,\ldots, y_n, z_1,\ldots, z_n$,
where $X = \{x_1,\ldots, x_n\}$ is the $n$-clique,
$Y = \{y_1,\ldots, y_n\}$, $Z = \{z_1,\ldots, z_n\}$ is the set of vertices
of degree $1$, and $y_i$ is the unique neighbor of $z_i$, for each $i$.

To show that $\gamma_t(G_k\Box G_n)\le 2kn+2k$, we will show that
$G_k\Box G_n$ has a total dominating set $D$ with $|D| = 2kn+2k$.
Set $$D = (A\times \{x_1,z_1\})\cup \Big(B\times \big((Y\setminus \{y_1\})\cup (Z \setminus \{z_1\})\big)\Big)\cup (C\times \{x_1,y_1\})\,,$$
see Fig.~\ref{TD-set}.

\begin{figure}[h!]
\begin{center}
\includegraphics[scale=1]{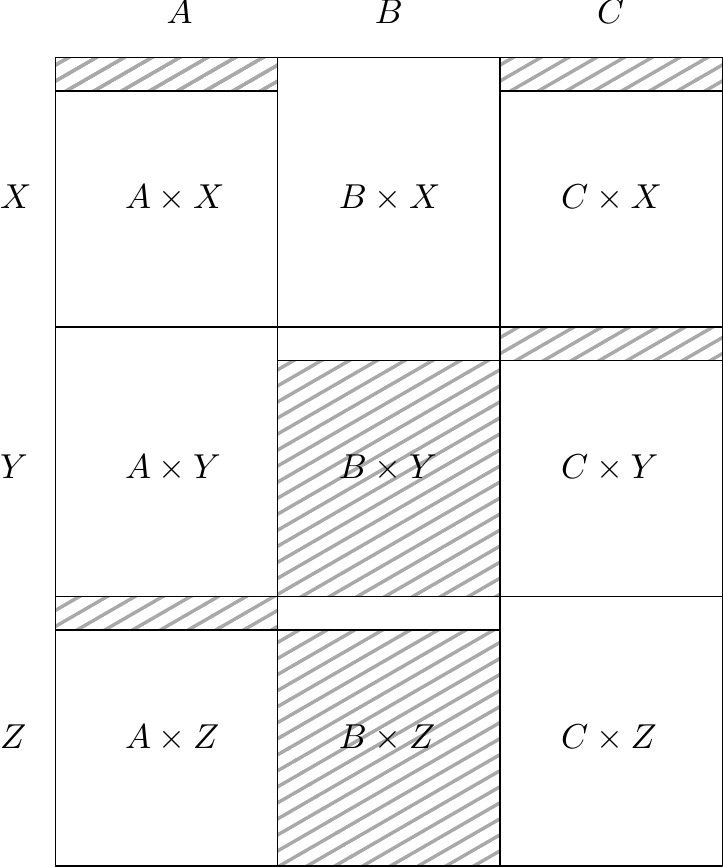}
\caption{The shaded area depicts the total dominating set $D$ of $G_k\Box G_n$.}
\label{TD-set}
\end{center}
\end{figure}
Clearly, $|D| = 2kn+2k$. To see that $D$ is a total dominating set of $G_k\Box G_n$, note that:
\begin{itemize}
  \item $A\times X$ is totally dominated by $A\times \{x_1\}$,
  \item $B\times \{x_1\}$ is totally dominated by $A\times \{x_1\}$ and
$B\times (X\setminus \{x_1\})$ is totally dominated by $B\times (Y\setminus \{y_1\})$,
  \item $C\times \{x_1\}$ is totally dominated by \barva{$C\times \{y_1\}$} and
$C\times (X\setminus \{x_1\})$ is totally dominated by $C\times \{x_1\}$,
  \item \barva{$A\times \{y_1\}$ is totally dominated by $A\times \{x_1\}$ and
$A\times (Y\setminus \{y_1\})$ is totally dominated by $B\times (Y\setminus \{y_1\})$,}
  \item $B\times \{y_1\}$ is totally dominated by $C\times \{y_1\}$ and
$B\times (Y\setminus \{y_1\})$ is totally dominated by $B\times (Z\setminus \{z_1\})$,
  \item $C\times \{y_1\}$ is totally dominated by $C\times \{x_1\}$ and
$C\times (Y\setminus \{y_1\})$ is totally dominated by $B\times (Y\setminus \{y_1\})$,
  \item \barva{$A\times \{z_1\}$ is totally dominated by $A\times \{z_1\}$ and
  $A\times (Z\setminus \{z_1\})$ is totally dominated by $B\times (Z\setminus \{z_1\})$,}
  \item $B\times \{z_1\}$ is totally dominated by $A\times \{z_1\}$ and
$B\times (Z\setminus \{z_1\})$ is totally dominated by $B\times (Y\setminus \{y_1\})$,
  \item $C\times \{z_1\}$ is totally dominated by $C\times \{y_1\}$ and
$C\times (Z\setminus \{z_1\})$ is totally dominated by $B\times (Z\setminus \{z_1\})$.
\end{itemize}

It remains to show that $\gamma_t(G_k\Box G_n)\ge 2kn+k$. Let $D$ be a minimum total dominating set of
$G_k\Box G_n$. Note that vertices of $C\times Z$ are totally dominated only by vertices in $(B\times Z)\cup (C\times Y)$
and no two vertices of $C\times Z$ have a common neighbor. Hence, at least $|C|\cdot |Z| = kn$ vertices from $D$
are needed to totally dominate $C\times Z$. Similarly, vertices in $B\times Z$ are totally dominated only by vertices in
$(A\times Z)\cup (B\times Y)\cup (C\times Z)$ and no two vertices of $B\times Z$ have a common neighbor.
Consequently, additional $|B|\cdot |Z| = kn$ vertices from $D$ are needed to totally dominate $B\times Z$.
Finally, vertices in $A\times X$ are totally dominated only by vertices in
$(A\times X)\cup (A\times Y)\cup (B\times X)$, and thus one can easily see that at least $k$ additional
vertices from $D$ are needed to totally dominate $A\times X$.
Altogether, the above arguments imply the claimed inequality, as $\gamma_t(G_k\Box G_n) = |D|\ge 2kn+k$.
\end{proof}
\end{sloppypar}

\begin{corollary}
For all integers $2\le k\le n$, we have
$$\frac{1}{2}+\frac{1}{4n}\le q_t(G_k, G_n)\le \frac{1}{2}+\frac{1}{2n}\,.$$
In particular, for every $k\ge 2$ we have
$$\lim_{n\to \infty} q_t(G_k, G_n) = \frac{1}{2}.$$
\end{corollary}

\barva{\section{Discussion on the total domination quotient}\label{sec:quot}}

We propose a further study of the quotient $q_t(G,H)$ for arbitrary graphs $G$ and $H$. In particular, it would be interesting to answer
the question about whether the quotient $q_t(G,H)$ equals $1/2$ only if one of the graphs is isomorphic to $K_2$. Note that the quotient can be arbitrarily large, as shown by $G$ and $H$ being complete graphs (in this case, $q_t(K_n,K_m) = \min\{m,n\}/4$). Moreover, recall the general bound $\gamma_t(G\Box H)\ge \rho_2(G)\gamma_t(H)$, cf.~\cite{bhr05}, where $\rho_2(G)$ denotes the $2$-packing number of a graph $G$, that is, the maximum number of pairwise disjoint closed neighborhoods of vertices in $G$. By this bound, we have that $q_t(G,H)\ge 1$ whenever $\gamma_t(G) = \rho_2(G)$.
Such graphs have been studied under the name $(\rho,\gamma_t)$-graphs and were
characterized by Dorfling et al.~\cite{dghm06}.

Next, we propose the following definition, in which $\cal G$ denotes the family of all connected graphs with no isolated vertices. Given a graph $G\in {\cal G}$, let $$q_t^{\rm{inf}}(G)=\inf_{H\in {\cal G}}\{q_t(G,H)\}\,.$$
That is, we want to express by this notion how close a graph $G\in \cal G$ can get to the bound from Theorem~\ref{thm:ho_ineq} when the other factor varies over all graphs in ${\cal G}$. For instance, by the above discussion, if $G$ is a $(\rho,\gamma_t)$-graph, then $q_t^{\rm{inf}}(G)\ge 1$. Clearly, if $G$ is one of the graphs from $\mathcal{F}_1 \cup \mathcal{F}_2\cup \mathcal{F}_3$, then $q_t^{\rm{inf}}(G)=1/2$ and the infimum is attained (hence, it is actually a minimum) when $K_2$ is chosen for $H$. Several questions naturally appear. For instance, is it true that for any graph $G$ there always exists a graph $H$ such that $q_t^{\rm{inf}}(G)=q_t(G,H)$? Note that the graphs $G_n$, for which we clearly have $q_t^{\rm{inf}}(G_n)=1/2$, all belong to the family $\mathcal{F}_1$, hence $q_t^{\rm{inf}}(G_n)=q_t(G_n,K_2)$.

We pose another question, which, if proven to have the positive answer, would considerably reduce the set of candidates $G$ and $H$ for which the equality $\gmtG \gmt (H) = 2 \gmtGH$ can hold.
\begin{question}
\label{q:1}
Is it true that for any connected graph $G$ with no isolated vertices we have $$q_t(G,H)\ge q_t(G,K_2),$$
where $H$ is an arbitrary connected graph with no isolated vertices?
In other words, is it true that $$\frac{\gmtGH}{\gmtH}\ge \frac{\gmt(\GcK)}{2},$$
for any graphs $G$ and $H$ in $\cal G$?
\end{question}
Note that the equality $\gmtG \gmt (H) = 2 \gmtGH$ holds if and only if $q_t(G,H)=1/2$, that is, if the total domination quotient of $G$ and $H$ attains the lower bound given by~\eqref{e2}. Hence, if the above question has affirmative answer, then the equality $\gmtG \gmt (H) = 2 \gmtGH$ implies $q_t(G,H)=q_t(G,K_2)=q_t(H,K_2)=1/2$. This in turn implies (by Theorem~\ref{thm:lu_hou_k2}) that both $G$ and $H$ belong to the family of graphs $\mathcal{F}_1 \cup \mathcal{F}_2\cup \mathcal{F}_3$. This would bring us closer to our suspicion that the pairs of graphs $G$ and $H$ from ${\cal G}$ such that $\gmtG \gmt (H) = 2 \gmtGH$ can only be found among the graphs from Theorem~\ref{thm:GH2} (in which case one of the factors is always $K_2$).

Furthermore, a positive answer to Question~\ref{q:1} would strengthen the bound of Ho from Theorem~\ref{thm:ho_ineq}, because reorganizing the inequality in Question~\ref{q:1} to $2{\gmtGH}\ge \gmt (\GcK){\gmtH}$ and using $\gmt(\GcK){\gmtH}\ge \gmtG\gmtH$, the truth of the first inequality implies $2{\gmtGH}\ge \gmtG\gmtH$. While we could not prove the inequality in Question~\ref{q:1}, we verified its truth by computer for all pairs of connected graphs $G$ and $H$ with no isolated vertices,
where $G$ has at most $8$ vertices and $H$ has at most $7$ vertices.

\section*{Acknowledgement}
We are grateful to anonymous reviewers for their helpful remarks and suggestions.

The research in this paper was supported in part by the Slovenian Research Agency (I0-0035, research programs P1-0285, P1-0297, research projects N1-0032, N1-0043, J1-5433, J1-6720, J1-6743, J1-7051, J1-7110 and two Young Researchers Grants).

\bibliographystyle{abbrv}

\end{document}